\documentclass[12pt]{article}
\usepackage{amssymb,amsthm,amsmath,amsfonts,latexsym,tikz,hyperref}
\usepackage[hmargin=.9in,vmargin=1in]{geometry}

\newtheorem{thm}{Theorem}[section]
\newtheorem{prop}[thm]{Proposition}
\newtheorem{cor}[thm]{Corollary}
\newtheorem{lem}[thm]{Lemma}
\newtheorem{conj}[thm]{Conjecture}
\newtheorem{exa}[thm]{Example}

\newcommand{\fp}{{\mathfrak p}}
\newcommand{\fP}{{\mathfrak P}}
\newcommand{\spc}{\ul{\rule{5pt}{0pt}}}
\newcommand{\tao}{\text{C1}}
\newcommand{\tat}{\text{C2}}
\newcommand{\sba}{\bar{s}}
\newcommand{\vsp}{\rule{0pt}{15pt}}

\DeclareMathOperator{\Pin}{Pin}
\DeclareMathOperator{\Pk}{Pk}
\DeclareMathOperator{\Val}{Val}

\newcommand{\ben}{\begin{enumerate}}
\newcommand{\een}{\end{enumerate}}
\newcommand{\ble}{\begin{lem}}
\newcommand{\ele}{\end{lem}}
\newcommand{\bth}{\begin{thm}}
\renewcommand{\eth}{\end{thm}}
\newcommand{\bpr}{\begin{prop}}
\newcommand{\epr}{\end{prop}}
\newcommand{\bco}{\begin{cor}}
\newcommand{\eco}{\end{cor}}
\newcommand{\bcon}{\begin{conj}}
\newcommand{\econ}{\end{conj}}
\newcommand{\bde}{\begin{defn}}
\newcommand{\ede}{\end{defn}}
\newcommand{\bex}{\begin{exa}}
\newcommand{\eex}{\end{exa}}
\newcommand{\barr}{\begin{array}}
\newcommand{\earr}{\end{array}}
\newcommand{\btab}{\begin{tabular}}
\newcommand{\etab}{\end{tabular}}
\newcommand{\beq}{\begin{equation}}
\newcommand{\eeq}{\end{equation}}
\newcommand{\bea}{\begin{eqnarray*}}
\newcommand{\eea}{\end{eqnarray*}}
\newcommand{\bal}{\begin{align*}}
\newcommand{\bce}{\begin{center}}
\newcommand{\ece}{\end{center}}
\newcommand{\bpi}{\begin{picture}}
\newcommand{\epi}{\end{picture}}
\newcommand{\bpp}{\begin{picture}}
\newcommand{\epp}{\end{picture}}
\newcommand{\bfi}{\begin{figure} \begin{center}}
\newcommand{\efi}{\end{center} \end{figure}}
\newcommand{\bprf}{\begin{proof}}
\newcommand{\eprf}{\end{proof}\medskip}
\newcommand{\capt}{\caption}
\newcommand{\bsl}{\begin{slide}{}}
\newcommand{\esl}{\end{slide}}
\newcommand{\bfr}{\begin{frame}}
\newcommand{\efr}{\end{frame}}

\newcommand{\comp}{\models}

\newcommand{\hqed}{\hfill \qed}

\newcommand{\eqed}[1]{$\textcolor{white}{\qed}\hfill{\dil#1}\hfill\qed$}

\newcommand{\ul}{\underline}
\newcommand{\ol}{\overline}

\newcommand{\hs}[1]{\hspace{#1}}
\newcommand{\hso}[1]{\hspace{-1pt}}
\newcommand{\vs}[1]{\vspace{#1}}

\newcommand{\emp}{\emptyset}

\newcommand{\sbs}{\subset}
\newcommand{\sbe}{\subseteq}

\newcommand{\spe}{\supseteq}

\newcommand{\case}[4]{\left\{\barr{ll}#1&\mbox{#2}\\#3&\mbox{#4}\earr\right.}

\newcommand{\flf}[2]{\left\lfloor\frac{#1}{#2}\right\rfloor}

\def\<{\langle}
\def\>{\rangle}

\newcommand{\ree}[1]{(\ref{#1})}

\newcommand{\ra}{\rightarrow}

\newcommand{\al}{\alpha}
\newcommand{\be}{\beta}

\newcommand{\si}{\sigma}

\newcommand{\De}{\Delta}

\newcommand{\bbN}{{\mathbb N}}
\newcommand{\bbP}{{\mathbb P}}

\newcommand{\cA}{{\cal A}}

\newcommand{\cB}{{\cal B}}

\newcommand{\cC}{{\cal C}}

\newcommand{\cI}{{\cal I}}

\newcommand{\cO}{{\cal O}}

\newcommand{\cV}{{\cal V}}

\newcommand{\fS}{{\mathfrak S}}

\newcommand{\ab}{\ol{a}}
\newcommand{\Ab}{\ol{A}}

\newcommand{\Sb}{\ol{S}}

\newcommand{\dil}{\displaystyle}

\begin{document}
\pagestyle{plain}

\title{Pinnacle Set Properties
}
\author{Rachel Domagalski\\[-5pt]
\small Department of Mathematics, Michigan State University,\\[-5pt]
\small East Lansing, MI 48824-1027, USA, {\tt domagal9@msu.edu}\\
Jinting Liang\\[-5pt]
\small Department of Mathematics, Michigan State University,\\[-5pt]
\small East Lansing, MI 48824-1027, USA, {\tt liangj26@msu.edu}\\
Quinn Minnich\\[-5pt]
\small Department of Mathematics, Michigan State University,\\[-5pt]
\small East Lansing, MI 48824-1027, USA, {\tt minnichq@msu.edu}\\
Bruce E. Sagan\\[-5pt]
\small Department of Mathematics, Michigan State University,\\[-5pt]
\small East Lansing, MI 48824-1027, USA, {\tt bsagan@msu.edu}\\
Jamie Schmidt\\[-5pt]
\small Department of Mathematics, Michigan State University,\\[-5pt]
\small East Lansing, MI 48824-1027, USA, {\tt schmi710@msu.edu}\\
Alexander Sietsema\\[-5pt]
\small Department of Mathematics, Michigan State University,\\[-5pt]
\small East Lansing, MI 48824-1027, USA, {\tt sietsem6@msu.edu}
}

\date{\today\\[10pt]
	\begin{flushleft}
	                                       \\[5pt]
	\end{flushleft}}

\maketitle

\begin{abstract}

Let $\pi=\pi_1\pi_2\ldots\pi_n$ be a permutation in the symmetric group $\fS_n$ written in one-line notation.  The pinnacle set of $\pi$, denoted $\Pin\pi$, is the set of all $\pi_i$ such that $\pi_{i-1}<\pi_i>\pi_{i+1}$.  This is an analogue of the well-studied peak set of $\pi$ where one considers values rather than positions.  The pinnacle set was introduced by Davis, Nelson, Petersen, and Tenner who showed that it has many interesting properties.  In particular, they proved that the number of subsets of $[n]=\{1,2,\ldots,n\}$ which can be the pinnacle set of some permutation is a binomial coefficient.  Their proof involved a bijection with lattice paths and was somewhat involved.  We give a simpler demonstration of this result which does not need lattice paths.  Moreover, we show that our map and theirs are different descriptions of the same function.  Davis et al.\ also studied the number of pinnacle sets with maximum $m$ and cardinality $d$ which they denoted by $\fp(m,d)$.  We show that these integers are ballot numbers and give two proofs of this fact: one using finite differences and one bijective.  Diaz-Lopez, Harris, Huang, Insko, and Nilsen found a summation formula for calculating the number of permutations in $\fS_n$ having a given pinnacle set.  We derive a new expression for this number which is faster to calculate in many cases.  We also show how this method can be adapted to find the number of orderings of a pinnacle set which can be realized by some $\pi\in\fS_n$.

\end{abstract}

\section{Introduction}

Let $\bbN$ and $\bbP$ be the nonnegative and positive integers, respectively. Given  $m,n\in\bbP$ we use the notation $[m,n]=\{m,m+1,\ldots,n\}$ for the interval they define.  We abbreviate $[1,n]$ to $[n]$.  Let $\fS_n$ be the symmetric group of all permutations $\pi=\pi_1\pi_2\ldots\pi_n$  of $[n]$ written in one-line notation.  An important statistic on $\fS_n$ is the {\em peak set} of a permutation $\pi$ which is defined as
$$
\Pk\pi=\{i \mid \pi_{i-1}<\pi_i>\pi_{i+1}\}\sbe[2,n-1].
$$
For example, if $\pi=18524376$ then $\Pk\pi=\{2,5,7\}$ since $\pi_2=8$, $\pi_5=4$, and $\pi_7=7$ are all bigger than the elements  directly to their left and right.
It is easy to see that $S\sbe[2,n-1]$ is the peak set of some $\pi\in\fS_n$ if and only if no two elements of $S$ are consecutive.  So the number of possible peak sets is a Fibonacci number.  One could also ask how many permutations have a given peak set.  This question was answered by Billey, Burdzy and Sagan.
\bth[\cite{bbs:pgp}]
\label{bbs}
If $n\in\bbP$ and $S\sbe[2,n]$ then 
$$
\#\{\pi \mid \Pk\pi=S\} = p(S;n) 2^{n-\#S-1}
$$
where $\#$ denotes cardinality and $p(S;n)$ is a polynomial in $n$ depending on $S$.\hqed
\eth

It is natural to study the values at the peak indices.  This line of research was initiated by Davis, Nelson, Petersen, and Tenner~\cite{dnpt:psp} and continued by 
Rusu~\cite{rus:spf}; Diaz-Lopez, Harris, Huang, Insko, and Nilsen~\cite{dhhin:fep}; and Rusu and Tenner~\cite{rt:apo}.  Define the {\em pinnacle set} of a permutation $\pi\in\fS_n$ to be
$$
\Pin\pi = \{\pi_i \mid \pi_{i-1}<\pi_i>\pi_{i+1}\}\sbe[3,n]
$$
Continuing with the example $\pi=18524376$ we see that $\Pin\pi=\{4,7,8\}$.  Following Davis et al., call a set $S$ an {\em admissible pinnacle set} if there is some permutation $\pi$ with $\Pin\pi=S$.  They found a criterion for $S$ to be admissible which will be useful in the sequel.  This result was stated in recursive fashion, but it is clearly equivalent to the following non-recursive version.
\bth[\cite{dnpt:psp}]
\label{s_i>2i}
Let $S=\{s_1<s_2<\ldots<s_d\}\sbs\bbP$. The set $S$ is an admissible pinnacle set if and only if  we have 
$$
s_i>2i
$$
for all $i\in[d]$.\hqed
\eth
\noindent Davis et al.\ were able to count the number of admissible pinnacle sets for $\pi\in\fS_n$.
\bth[\cite{dnpt:psp}]
\label{binom}
If 
$$
\cA_n=\{S \mid \text{$S=\Pin\pi$ for some $\pi\in\fS_n$}\}
$$
then

\eqed{\#\cA_n=\binom{n-1}{\flf{n-1}{2}}.}
\eth
\noindent They also studied the more refined constants
$$
\fp(m,d) = \#\{S\in\cA_n \mid \max S = m \text{ and } \#S=d\}
$$
where $n\ge m$.  Note that if $S=\Pin\pi$ for some $\pi\in\fS_n$ then $S$ is also a pinnacle set of some $\pi'\in\fS_{n'}$ for all $n'\ge n$ since one can just add values larger than $n$ to the beginning of $\pi$ in decreasing order.  It follows that the exact value of $n$ does not play a role in the definition of 
$\fp(m,d)$.

A number of questions have been raised about pinnacle sets.  For example, if
$$
p_S(n) =\#\{\pi\in\fS_n \mid \Pin\pi=S\}
$$
then how can one compute these numbers as there does not seem to be an analogue of Theorem~\ref{bbs} in the context of pinnacles.  Davis et al.\ gave a recursive procedure for doing so, and then a non-recursive summation formula for determining the $p_S(n)$ was proposed in the paper of Diaz-Lopez et al.  

Another problem suggested earlier is as follows.  Given an admissible $S$, a permutation $\si$ of $S$ is called an {\em admissible ordering} if there is a  $\pi\in\fS_n$ with $\Pin\pi=S$ and the pinnacles of $\pi$ occur in the same order as they do in $\si$.  Let 
$$
\cO(S) = \{\si \mid \text{$\si$ is an admissible ordering of $S$}\}.
$$
For example, if $S=\{3,5,7\}$ then
$\si=537\in\cO(S)$ as witnessed by 
$\pi=4513276$.  But $375\not\in\cO(S)$ since in order for $6$ not to be a pinnacle, it must be directly to the left or right of $7$ and both choices lead to a contradiction.
The set $\cO_S$ was studied in the articles of Rusu, and  of Rusu and Tenner.  In the latter paper, the authors asked for a function to compute $\#\cO(S)$.

The rest of this article is structured as follows.  The proof of Theorem~\ref{binom} in~\cite{dnpt:psp} used a bijection involving lattice paths and was somewhat complicated.  In the next section we provide a simpler demonstration which does not need lattice paths.  In fact, we show that our map and theirs are different descriptions of the same function.
Section~\ref{bn} is devoted to the $\fp(m,d)$.  We will show that these are actually just ballot numbers, and do this in two ways: using the theory of finite differences and via a bijection.  In the final section we give a new summation formula for calculating $p_S(n)$ which is faster in many cases than the one proposed in~\cite{dhhin:fep}.  We also show that our algorithm can be modified to find $\#\cO(S)$.

\section{Counting admissible pinnacle sets}
\label{cap}

In this section we  give our proof of Theorem~\ref{binom}.  Our strategy will be as follows.  First, we will introduce the set of interleaved permutations which are obviously counted by the desired binomial coefficient.  Next, we will associate with each admissible pinnacle set $S$ a particular permutation $\pi$ such that $\Pin\pi = S$.  This permutation will be called right canonical because its pinnacles will be as far right as possible.  Finally, we will show that the set of interleaved permutations and the set of right canonical permutations are, in fact, the same.  This will complete the proof of the theorem.

An {\em interleaved permutation}  $\pi\in\fS_n$ is one constructed in the following manner.  Pick any $A\sbe[2,n]$ with $\#A=\flf{n-1}{2}$. 
\ben
\item[I1] Fill the first $\flf{n-1}{2}$ even positions of $\pi$ with the elements of $A$ in increasing order.
\item[I2] Fill the remaining positions of $\pi$ with the elements of $\Ab=[n]-A$ in increasing order. 
\een
As an example, suppose $n=9$ and $A=\{2,3,7,9\}$.  After step I1 we have
$$
\pi = \spc\ 2\ \spc\ 3\ \spc\ 7\ \spc\ 9\ \spc.
$$
Since $\Ab=\{1,4,5,6,8\}$, after I2 we have the full interleaved permutation
\begin{equation}
\label{inter} 
\pi = 1\ 2\ 4\ 3\ 5\ 7\ 6\ 9\ 8.
\end{equation}
Let 
$$
\cI_n =\{\pi\in\fS_n \mid \text{$\pi$ is interleaved}\}.
$$
Clearly $\pi\in\cI_n$ is completely determined by the choice of $A$.  It follows immediately that
\begin{equation}
\label{cI}
\#\cI_n = \binom{n-1}{\flf{n-1}{2}}.
\end{equation}

Now given an admissible pinnacle set $S=\{s_1<s_2<\ldots<s_d\}\sbs[n]$ we wish to construct a permutation $\pi\in\fS_n$ with $\Pin\pi=S$.  We use the following algorithm to construct the {\em right canonical permutation} $\pi$ from $S$.  We first deal with the case where $n$ is odd. Let $\Sb=[n]-S$.
\ben
\item[C1] Place elements of $\Sb$ in $\pi$ moving right to left, starting with the largest unused element of $\Sb$ and then decreasing until an element less than the largest unused element of $S$ is placed.
\item[C2]  Place the largest unused element of $S$ in the rightmost unused position.
\item[C3]  Iterate C1 and C2 until all elements of $S$ and $\Sb$ are placed.
\een
If $n$ is even, the only change to this procedure is that we fill both $\pi_n$ and $\pi_{n-1}$ with elements of $\Sb$ before considering whether to place an element of $S$.  To illustrate, consider $n=9$ and $S=\{4,7,9\}$.  So $\Sb=\{1,2,3,5,6,8\}$.  Here is the construction of $\pi$ where, at each stage, we note whether C1 or C2 is being used.
$$
\barr{r|ccccccccc}
\text{step}&\tao&\tat&\tao&\tat&\tao&\tao&\tat&\tao&\tao\\
\pi & 8 & 98 & 698 & 7698 & 57698 & 357698 & 4357698 & 24357698
& 124357698
\earr
$$
So the right canonical permutation for $S=\{4,7,9\}$ is
$\pi=124357698$.  Note that $\Pin\pi=S$.  Furthermore, this is the same permutation as obtained in~\ree{inter}.  However, neither the sets $A$ nor $\Ab$ equals $S$.
Let
$$
\cC_n =\{\pi\in\fS_n \mid \text{$\pi$ is right canonical}\}.
$$

We first need to show that C1--C3 is well defined in that every position of $\pi$ gets filled and that we always have
$\Pin\pi=S$.
\begin{lem}
\label{C1--C3}
If $S\sbs[n]$ is an admissible set then C1--C3 produces a permutation $\pi$ with  $\Pin \pi=S$.  Thus
$$
\#\cC_n=\#\cA_n.
$$
\end{lem}
\begin{proof}
Clearly the second sentence follows from the first.
For the first sentence, we will present details for the case when $n$ is odd.  If $n$ is even, then one can just place the largest element of $\Sb$ in position $n$ and proceed as in the odd case.

The following notation will be useful.  Let
\begin{align*}
    S & = \{s_1<s_2<\ldots<s_d\},\\
    \Sb&=\{\sba_1<\sba_2<\ldots<\sba_{n-d}\}.
\end{align*}
We will also let $S_p$ and $\Sb_p$ denote the elements of $S$ and of $\Sb$, respectively, which have not been used during the placement of $\pi_n,\pi_{n-1},\ldots,\pi_p$.

We will use reverse induction on the position $p$ being filled in $\pi$.  When $p=n$, we have $\Sb\neq\emp$ since $1$, which is always a non-pinnacle, must be in $\Sb$.  So there is an element $\sba_{n-d}$ to place in position $n$.  Furthermore this element can not be a pinnacle since it is the last element of the permutation, which agrees with the fact that it is in $\Sb$.

Suppose that $\pi_n,\pi_{n-1},\ldots,\pi_{p+1}$ have been constructed.  Suppose first that $\pi_{p+1}\in\Sb$.  One subcase is if either $S_p=\emp$, or $S_p\neq\emp$ and $\pi_{p+1}>\max S_p$.  We must show that $\Sb_p\neq\emp$ so that we can let
$\pi_p=\max\Sb_p$.  This is true when $S_p=\emp$  since 
$|S_p\uplus \Sb_p|=p$.  If the second option holds then we have $\pi_{p+1}>\max S_p$.  But there must be at least  two elements of $\Sb$ smaller than $\max S_p$ since $S$ is admissible and so there is some permutation making  $\max S_p$ a pinnacle.
Also, these elements must still be in $\Sb_p$ since elements of this set are placed in decreasing order right to left.  Thus this set is nonempty as desired. Furthermore, $\pi_p$ is not a pinnacle since it is smaller than $\pi_{p+1}$.

Now consider the subcase when $\pi_{p+1}<\max S_p$.  Then we let $\pi_p=\max S_p$ which is well defined.  But we must show that $\pi_p$ is a pinnacle.  We know $\pi_p>\pi_{p+1}$.  So there remains to check whether one can construct 
$\pi_{p-1}$  with $\pi_{p-1}<\pi_p$. 
For this, it suffices to show that $\Sb_{p-1}\neq\emp$ since then we will have
$\pi_{p-1}=\max\Sb_{p-1}<\pi_{p+1}<\pi_p$.
Note that this will also finish the induction step. 

We claim that if $\pi_p=s_i$ and $\pi_{p+1}=\sba_j$ then  $j>i$.  It will then follow that $\sba_{j-1}$ exists and can be used for $\pi_{p-1}$.  But by Theorem~\ref{s_i>2i} we have $s_i>\sba_{i+1}$.
Indeed, if $s_i<\sba_{i+1}$ then at most the elements $s_1,\ldots,s_{i-1},\sba_1,\ldots,\sba_i$ are less than $s_i$ so that $s_i\le 2i$,  a contradiction.
Also, elements of $\Sb$ are placed in decreasing order with $s_i$ being placed as early as possible with a smaller element to its right. The desired bound on $j$ follows.
\end{proof}

We are now ready to give our proof of Theorem~\ref{binom}.
\bth
We have $\cC_n=\cI_n$.  Thus
$$
\#\cA_n=\binom{n-1}{\flf{n-1}{2}}.
$$
\eth
\begin{proof}
The second statement follows directly from the first, Lemma~\ref{C1--C3}, and equation~\ree{cI}.
So we only need to prove that the two sets are the same.  We will consider the case when $n$ is odd, as the even case is similar.

We begin by showing that any right canonical permutation $\pi$ is interleaved.  That is to say, the subword consisting of all even indices is an increasing sequence, and the subword consisting of all odd indices is an increasing sequence starting with 1.

In terms of the placement of $1$, note that $\pi_1$ is not a  pinnacle.  And since non-pinnacles are placed in decreasing order from right to left we must have $\pi_1=1$.

To finish this direction, it is enough to show that for any elements $\pi_i$ and $\pi_{i+2}$, we have that $\pi_{i+2}>\pi_i$. Note that we are done immediately if $\pi_i$ and $\pi_{i+2}$ are either both pinnacles or both non-pinnacles since the construction places them in decreasing order from right to left. 
If $\pi_{i+2}$ is a pinnacle and $\pi_i$ is not, then by the pinnacle assumption $\pi_{i+2}>\pi_{i+1}$.  And since non-pinnacles are placed in decreasing order right to left
$\pi_{i+1}>\pi_i$.  Combining the two inequalities gives the desired result.
Finally, suppose $\pi_i$ is a pinnacle and $\pi_{i+2}$ is not.  Then $\pi_{i+1}$ is not a pinnacle, being adjacent to $\pi_i$.  And, by construction, $\pi_{i+1}$ must be the first available non-pinnacle right to left which is smaller than $\pi_i$.  It follows that $\pi_{i+2}>\pi_i$.

For set containment the other way, let $\pi$ be an interleaved permutation. 
It suffices to show that if the elements  of $\pi$ are placed right to left then they follow C1--C3.  
Consider $\pi_i$ placed after $\pi_{i+1}$ with $1<i<n$.  The boundary cases when $i=1$ or $n$ are similar.
If $\pi_i<\pi_{i+1}$ then $\pi_i$ is a non-pinnacle and $\pi_{i+1}$ is either a non-pinnacle or a pinnacle.  In the first case, the non-pinnacles are being placed in decreasing order as desired.  In the second, the previously placed non-pinnacle is $\pi_{i+2}$.  So the same conclusion holds by the interleaving condition. Now consider the possibility $\pi_i>\pi_{i+1}$.  By the interleaving condition, $\pi_{i-1}<\pi_{i+1}$ so $\pi_i$ is a pinnacle.
Either $\pi_{i+2}$ is a pinnacle or not, the latter possibility including the case that $\pi_{i+2}$ does not exist.  If it is, then the interleaving condition shows that pinnacles are being placed in decreasing order.  
If $\pi_{i+2}$ is not a pinnacle, then this fact and the  interleaving condition again  imply $\pi_i<\pi_{i+2}<\pi_{i+3}$.  It follows that $\pi_i$ was placed after the first smaller non-pinnacle and, by the interleaving condition one last time, that any pinnacles to its right are larger.  This completes the proof of the other containment.
\end{proof}

Given a set $A$ and $k\in\bbN$ we let $\binom{A}{k}$ be the set of all $k$-element subsets of $A$.  The above construct gives us a bijection
$$
\psi:\binom{[2,n]}{\flf{n-1}{2}}\ra\cA_n
$$
given by
$$
\psi(A) = \Pin\pi
$$
where $\pi$ is the interleaving permutation corresponding to $A$. 

\bfi
\begin{tikzpicture}
\draw[->] (0,-3)--(0,3);
\draw[->] (0,0)--(9,0);
    	\foreach \x in {1,...,8}
     		\draw (\x,1pt) -- (\x,-3pt)
			node[anchor=north] {\x};
    	\foreach \y in {-2,...,2}
     		\draw (1pt,\y) -- (-3pt,\y) 
     			node[anchor=east] {\y};
\fill(0,0) circle(.1);
\fill(1,-1) circle(.1);
\fill(2,-2) circle(.1);
\fill(3,-1) circle(.1);
\fill(4,0) circle(.1);
\fill(5,1) circle(.1);
\fill(6,0) circle(.1);
\fill(7,1) circle(.1);
\fill(8,0) circle(.1);
\draw(.3,-.7) node{$2$};
\draw(1.3,-1.7) node{$3$};
\draw(2.7,-1.7) node{$4$};
\draw(3.7,-.7) node{$5$};
\draw(4.3,.7) node{$6$};
\draw(5.7,.7) node{$7$};
\draw(6.3,.7) node{$8$};
\draw(7.7,.7) node{$9$};
\draw (0,0)--(2,-2)--(5,1)--(6,0)--(7,1)--(8,0);
\end{tikzpicture}
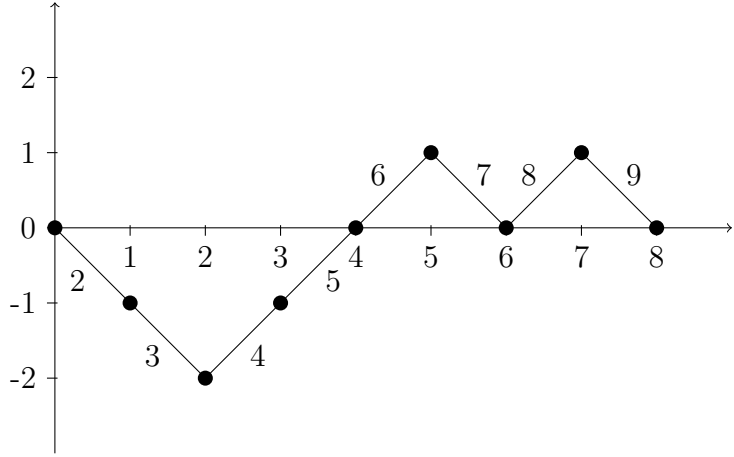
\capt{The lattice path $L$ for $A=\{2,3,7,9\}$ \label{path}}
\efi

In~\cite{dnpt:psp}, the authors proved Theorem~\ref{binom} using a bijection 
$$
\phi:\binom{[2,n]}{\flf{n-1}{2}}\ra\cA_n
$$
defined as follows.
An {\em up-down lattice path} $L$ starts at the origin and uses steps which are either up ($U$) or down ($D$) parallel to the vectors $[1,1]$ and $[1,-1]$, respectively.  
For more information about lattice paths, see the text of Sagan~\cite{sag:aoc}.
It will be convenient to index the steps of $L$ with $[2,n]$ and write
$L=s_2 s_3\ldots s_n$.  Associate with $A\in \binom{[2,n]}{\flf{n-1}{2}}$ the lattice path $L$ such that
$$
s_i =\case{D}{if $i\in A$,}{U}{if $i\not\in A$.}
$$
To illustrate, if $n=9$ and $A=\{2,3,7,9\}$ as in the example beginning this section then 
$$
L=DDUUUDUD
$$
as depicted in Figure~\ref{path} where each step is labeled by its index.  We now define
$$
\phi(A) = \{i \mid \text{in $L$ either $s_i=U$  strictly below the $x$-axis, or $s_i=D$ weakly above the $x$-axis}\}.
$$
Continuing our example, 
$\phi(\{2,3,7,9\})=\{4,7,9\}=\psi(\{2,3,7,9\})$.  This is not an accident.
\bpr
We have $$\phi=\psi.$$
\epr
\begin{proof}
We will give the proof for $n$ odd as the even case is similar.  Let $l=(n-1)/2$.  We need to show that $\phi(A)=\psi(A)$ for all 
$A\in \binom{[2,n]}{l}$.  Suppose
$A=\{a_1<a_2<\ldots<a_l\}$ and 
$\Ab=[n]-A=\{\ab_1<\ab_2<\ldots<\ab_{n-l}\}$.
Let $L$ and $\pi$ be the lattice path and interleaved permutation, respectively, associated with $A$.   So $\psi(A)=\Pin\pi$
and there will be two cases depending on whether a pinnacle of $\pi$ comes from $A$ or $\Ab$ 

In the first case, suppose $a_i\in\Pin\pi$. Since $\pi$ is interleaved, this is equivalent to $a_i=\pi_{2i}>\pi_{2i+1}=\ab_{i+1}$.
 Recall that $a_i$ indexes the $i$th $D$ step of $L$, and similarly for $\ab_{i+1}$ and $U$ steps.
So the previous inequality is equivalent to step $s_{a_i}=D$ being preceded by more up steps than down steps.  And this is precisely the condition for  $a_i$ to be the index of a down step weakly above the $x$-axis, which means it is in $\phi(A)$.  Thus this case is complete.

In a similar manner, one proves that $\ab_i\in\Pin\pi$ if and only if $\ab_i$ is the index of an up step strictly below the $x$-axis.  This completes the second case and the proof.
\end{proof}

\section{Ballot numbers}
\label{bn}

Davis et al.\ derived a number of properties of the constants $\fp(m,d)$ which count the number of admissible pinnacle sets $S$ with $d$ elements and maximum $m$.  In this section we prove that these constants are, in fact, ballot numbers.  We give two proofs of this result.  In the first, we derive a formula for $\fp(m,d)$ using finite differences and then show that it agrees with the well-known expression for ballot numbers.  In the second, we give an explicit bijection between these admissible sets and ballot sequences.

Suppose we are given nonnegative integers $p>q$.  A {\em $(p,q)$ ballot sequence} is a permutation 
$\be=\be_1\be_2\ldots\be_{p+q}$
of $p$ copies of the letter $X$ and $q$ copies of the letter $Y$ such that in any nonempty prefix $\be_1\be_2\ldots\be_i$ the number of $X$'s is greater than the number of $Y$'s.
Let
$$
\cB_{p,q}=\{\be \mid \text{$\be$ is a $(p,q)$ ballot sequence}\}.
$$
The following result is well known.
\bth[\cite{and:sdp},\cite{ber:sp}]
\label{ballot}
For nonnegative integers $p>q$ we have

\vs{10pt}

\eqed{
\#\cB_{p,q}= \frac{p-q}{p+q} \binom{p+q}{q}.
}
\eth
Note that if we let $p=d+1$ and $q=d$ then the previous result gives
get
$$
\#\cB_{d+1,d} =\frac{1}{2d+1}\binom{2d+1}{d} = C_d
$$
where $C_d$ is the $d$th Catalan number.

Our first proof that the $\fp(m,d)$ are ballot numbers will use the theory of finite differences.  If $f(m)$ is a function of a nonnegative integer $m$ then its {\em forward difference} is the function $\De f$ defined by
$$
\De f(m) = f(m+1)-f(m).
$$
For a fixed $d\in\bbP$, define the following polynomial in $m$ of degree $d-1$
$$
p_d(m) = \frac{m-2d+1}{(d-1)!}\prod_{i=2}^{d-1}(m-i).
$$

\begin{lem}
\label{Delta}
The polynomial $p_d(m)$ satisfies 
$$
\Delta p_d(m) = p_{d-1}(m)
$$
and 
$$
p_d(2d+1) = C_d.
$$
\end{lem}

\begin{proof}
To prove the first equality, we compute
\begin{align*}
    \Delta p_d(m) & = p_d(m+1)-p_d(m)\\[5pt]
   & = \frac{m-2d+2}{(d-1)!}\ \prod_{i = 2}^{d-1}(m+1-i)- \frac{m-2d+1}{(d-1)!}\ \prod_{i = 2}^{d-1}(m-i)\\[5pt]
    & = \frac{(m-2d+2)(m-1)-(m-2d+1)(m-d+1)}{(d-1)!}\ \prod_{i = 2}^{d-2}(m-i)\\[5pt]
    & = \frac{(d-1)(m-2d+3)}{(d-1)!}\ \prod_{i = 2}^{d-2}(m-i)\\[5pt]
    & = p_{d-1}(m).
\end{align*}
For the second equality, we have
\begin{align*}
    p_d(2d+1)
    & = \frac{2}{(d-1)!}\ \prod_{i = 2}^{d-1}(2d+1-i)\\[5pt]
    & = \frac{2d}{d!}\cdot \frac{(2d-1)!}{(d+1)!}\\[5pt]
    & = \frac{(2d)!}{d!(d+1)!}\\[5pt]
    & = C_d
\end{align*}
which finishes the proof. 
\end{proof}

Note that by the criterion in Theorem~\ref{s_i>2i}, $\fp(m,d)$ can only be nonzero if $m>2d$.
\bth
If $m,d\in\bbP$ with $m>2d$ then $\fp(m,d)=p_d(m)$.  Thus
$$
\fp(m,d)=\frac{m-2d+1}{m-1}\binom{m-1}{d-1}=\#\cB_{m-d,d-1}.
$$
\eth
\begin{proof}
Induct on $d$ where the base case of $d=1$ is trivial to verify.  To finish the
first claim, it suffices to use the previous lemma and show that both
$\Delta \fp(m,d) = \fp(m,d-1)$
and $\fp(2d+1,d) = C_d$.  But these were proved in~\cite[Sections 2.2--2.3]{dnpt:psp}.  The first displayed equality now follows from simple manipulation of the definition of $p_d(m)$, while the second comes from Theorem~\ref{ballot}.
\end{proof}

We would like to give a bijective proof of the relationship between admissible pinnacle sets and ballot sequences from the previous theorem.  Let
$$
\fP(m,d)=\{S \mid \text{$S$ admissible with $\max S = m$  and $\#S=d$}\}
$$
so that $\#\fP(m,d)=\fp(m,d)$.
For $m>2d$, define a map
$$
\eta:\cB_{m-d,d-1}\ra\fP(m,d)
$$
by sending ballot sequence $\be=\be_1\be_2\ldots\be_{m-1}$ to
$$
\eta(\be)=\{i \mid \be_i=Y\} \uplus \{m\}.
$$
For example, if $m=9$, $d=3$ and $\be=XXXYXXYX$ then
$$
\eta(\be)=\{4,7\}\uplus\{9\} = \{4,7,9\}.
$$
\bth
The map $\eta$ is a well-defined bijection.
\eth
\bprf
We must first show that $\eta$ is well defined in that $\eta(\beta)\in\fP(m,d)$.  Since $\be\in\cB_{m-d,d-1}$ we see that the set
$\{i \mid \be_i=Y\}$ is contained in $[m-1]$ and has cardinality $d-1$.  It follows that $S=\eta(\be)$ has maximum $m$ and cardinality $d$.  

There remains to show that $S=\{s_1<s_2<\ldots<s_d\}$  is admissible.
By Theorem~\ref{s_i>2i}, it suffices to show that $s_i>2i$ for all $i$.  But $s_i$ is the index of the $i$th $Y$ in $\be$.  Since $\be$ is a ballot sequence, this $Y$ is preceded by $i$ copies of $Y$ (including itself) and at least $i+1$ copies of $X$.  
So $s_i\ge i+(i+1)=2i+1$ which is what we wished to prove.

To show that $\eta$ is a bijection, we create its inverse.
Given $S\in\fP(m,d)$ we define $\eta^{-1}(S)=\be=\be_1\be_2\ldots\be_{m-1}$ by letting
$$
\be_i = \case{X}{if $i\not\in S$,}{Y}{if $i\in S$.}
$$
The proof that $\eta^{-1}$ is well defined is similar to the one for $\eta$.  And proving that the compositions of $\eta$ with $\eta^{-1}$ are identity maps is easy.  So we are done.
\eprf

\section{Permutations with a given pinnacle set}
\label{pgp}

Given an admissible set $S$, there does not seem to be an expression for $p_S(n)$, the number of permutations in $\fS_n$ with $S$ as pinnacle set, analogous to the one in Theorem~\ref{bbs} for peak sets.  In~\cite{dnpt:psp}, they found expressions for $p_S(n)$ when $\#S\le 2$ as well as bounds for general $S$, and asked whether an exact formula could be given in the general case.  Such an expression was given in~\cite{dhhin:fep} as a summation.  In this section we will give another sum which is asymptotically more efficient.  In addition, this method can be extended to count $\#\cO(S)$, the number of admissible orderings of $S$.

Since our sum will involve a significant amount of new notation, we will collect it here and then explain its relevance afterwards.
Fix $n\in\bbP$. Suppose we have an admissible pinnacle set $S = \{s_1 < s_2 < \ldots < s_d\}$ for permutations in $\fS_n$. We use the convention $s_0 = 0$ and $s_{d+1} = n+1$ and let
$$n_i = s_{i+1}-s_{i}-1$$ 
for $0\le i\le d$.  Let
$$
D = \{1_l, 1_r, 2_l,2_r,\ldots, d_l, d_r\}
$$
and give the following total order to $D$'s elements
$$
1_l< 1_r<2_l< 2_r<\ldots< d_l< d_r.
$$
We call $i_l$ and $i_r$ the elements of {\em rank $i$} in $D$.
If $B\sbe D$ then we will let
$$
b=\#B
$$
and
$$
r_j =\text{the rank of the $j$th smallest element of $B$}
$$
for $1\le j\le b$.  We also define
$$
b_i  = \text{the number of elements in  $B$ with rank at least $i$}.
$$
Note that we always have $b_1 = b$
and $b_{d+1}=0$ since $d$ is the largest rank. For example, if $d=4$, then $D = \{1_l, 1_r, 2_l, 2_r, 3_l, 3_r, 4_l, 4_r\}$ and one possible $B$ might be $B = \{1_l, 3_l,3_r, 4_r\}$ which has $r_1 = 1, r_2 = 3, r_3 = 3, r_4 = 4$ and $b_1 = 4, b_2 = 3, b_3 = 3, b_4 = 1, b_5=0$.  We can now state the first main result of this section.

\begin{thm}
\label{main1}
Given $n\in\bbP$ and admissible  $S=\{s_1<s_2<\ldots<s_d\}$ we have
$$
p_S(n)=2^{n-2d-1}\sum_{B \subseteq D:\ |B|\leq d} (-1)^{b} (d-b)! \left(\prod_{i=0}^{b-1} \left( d+1-i-r_{b-i}\right)\right) \left(\prod_{i=0}^{d} (d+1-i-b_{i+1})^{n_i}\right).
$$
\end{thm}

To prove this, it will be convenient to convert the linear permutations we have been studying into cyclic ones in order to avoid considering boundary cases.
Given a linear permutation $\pi=\pi_1\pi_2\ldots\pi_n$ the corresponding {\em cyclic permutation} is the set of permutations
$$
[\pi]=\{\pi_1\pi_2\ldots\pi_n,\hs{10pt} \pi_2\ldots\pi_n\pi_1,\hs{10pt} \ldots,\hs{10pt} \pi_n\pi_1\ldots\pi_{n-1}\}.
$$
Intuitively, we think of $[\pi]$ as the result of arranging the elements of $\pi$ on a circle.  
Let 
$$
[\fS_n]=\{[\pi] \mid \pi\in\fS_n\}.
$$
For example if $\pi=1324$ then
$$
[\pi]= \{1324,\ 3241,\ 2413,\ 4132\}.
$$
We are also using the bracket notation in $[n]$ where $n\in\bbN$ but this should not cause any confusion.  Cyclic permutations are of interest in part because of their relation with pattern avoidance, standard Young tableaux, quasisymmetric functions, and other mathematical objects~\cite{agrr:cqf,cal:pac,dlmsss:cpca,dlmsss:csc,glw:pcc1,glw:pcc2}.

We define the {\em pinnacle set} of $[\pi]=[\pi_1\pi_2\ldots\pi_n]$ to be 
$$
\Pin[\pi]=\{\pi_i \mid \text{$\pi_{i-1}<\pi_i>\pi_{i+1}$ where subscripts are taken modulo $n$}\}.
$$
Continuing our example from the last paragraph
$$
\Pin[1324]=\{3,4\}.
$$
Note in particular that $\Pin[12]=\{2\}$ and, more generally, $n\in\Pin[\pi]$ for any $[\pi]\in[\fS_n]$ where $n\ge2$.
\begin{lem}
\label{circ}
For $n\in\bbP$, there is a bijection between linear permutations in $\fS_n$ with pinnacle set $S$ and cyclic permutations  in
$[\fS_{n+1}]$ with 
pinnacle set $S' = S \cup \{n+1\}$.
\end{lem}
\begin{proof}
Given a linear $\pi$, append the element $n+1$ to the end of $\pi$ 
and take the corresponding equivalence class in $\fS_{n+1}$ to form an element of $[\fS_{n+1}]$.
The map is clearly invertible and does not destroy or create any pinnacles for elements in $[n]$.  Since $n+1\ge2$, we know that $n+1$ will become a pinnacle. Therefore the map has the desired properties concerning the pinnacle set.
\end{proof}

Consider some admissible pinnacle set $S = \{s_1 < s_2 < \ldots < s_d\}$. Given the above lemma, we may count the number of permutations in $\fS_n$ with pinnacle set $S$ by counting the number of cyclic permutations $[\pi] \in [\fS_{n+1}]$ with pinnacle set $S' = S \cup \{n+1\}$ where we let $s_{d+1} = n+1$. Therefore, much of what follows will be in regards to cyclic permutations with pinnacle set $S'$.

A {\em factor} of a (cyclic) permutation is a subsequence of consecutive elements.
We may attempt to construct a $[\pi]$ with pinnacle set $S'$ by first putting the elements of $S'$ in some cyclic order, and then placing all elements in $\overline{S'}  = [n+1]- S'$ into either decreasing factors starting with some $s_i$, or into increasing factors ending with  some $s_i$. 
Such a $[\pi]$ will then be completely determined by the increasing/decreasing factors that each element of $\overline{S'}$ falls into, and we will call every such  assignment a \textit{placement}. Note that it is possible for multiple placements to result in the same permutation since each vale (an element of $[\pi]$ smaller than the elements on either side) can be part of the factor on either side. For example, start with a desired pinnacle set $\{4,5\}$  and  place non-pinnacles between these  elements to form the cyclic permutation $[\pi]=[14325]$. Then $[\pi]$ would be associated with a placement where the decreasing factor starting with $4$ is $43$ and the increasing factor ending with $5$ is  $25$.  But it would also be associated with a placement having these factors be $432$ and $5$, respectively.

It is also possible, depending on the placement, that $[\pi]$ will not have pinnacle set $S'$ if no sufficiently small elements are placed between two pinnacles. In our example above, this could have happened if we had placed $1$, $2$ and $3$ all in the increasing factor ending in $5$, resulting in the cyclic permutation $[41235]$ in which only $5$ is a pinnacle. It is true, however, that any $[\pi]$ so constructed will have a pinnacle set that is a subset of $S'$ since every non-pinnacle was placed so that its factor contains an $s_i$ which is the largest element. For our arguments, we will focus on counting placements and then convert them into permutations later. 

Fix a cyclic ordering of the pinnacle indices and write it as $[\tau] = [\tau_1\cdots\tau_{d+1}] \in [\fS_{d+1}]$. An example is shown in Figure~\ref{DalFig} where $\tau=[7612354]$.
Now given a placement consistent with this ordering, for every space between two adjacent elements in $[\tau]$ define the {\em dale set} of this placement to consist of all elements between the two corresponding pinnacles that are also smaller than both pinnacles. 
So in Figure~\ref{DalFig} the dales are outlined by triangles with solid lines as sides.
If $s_i$ is the smaller of the two pinnacles, then we say that the dale has \textit{rank $i$}. Note that the rank is from the index of $s_i$ and not its actual value.  We will further denote the rank as either $i_l$ or $i_r$ depending on whether the dale is to the left, or right of the pinnacle $s_i$. In Figure~\ref{DalFig} the dale ranks are given along the $x$-axis. Define the \textit{dale rank set} $D_{[\tau]}$ to be the set of the dale ranks of $[\tau]$.  And define the \textit{master dale rank set} to be 
$$
D = \{1_l, 1_r, 2_l,2_r,\ldots, d_l, d_r\}
$$
so that $D\spe D_{[\tau]}$ for all $[\tau]$.
 In Figure \ref{DalFig}, we have that $D_{[\tau]} = \{1_l, 1_r, 2_r, 3_r, 4_1, 4_r, 6_l\}$ while $D = \{1_l, 1_r, 2_l, 2_r, \ldots, 6_l, 6_r\}$. Note that, by our definitions, there will be no dales in the case where $d=0$.

Clearly $D_{[\tau]}$ will be a subset of $D$ consisting of exactly $d+1$ elements if $d>0$, and empty otherwise. We can derive further information about $D_{[\tau]}$ if we want, such as how it will always contain both $1_l$ and $1_r$ if $d>0$, how it will never contain both $d_l$ and  $d_r$ if $d>1$, and how  $D_{[\tau]}$ will never be able to have certain combinations of the higher ranked dales. These facts are not necessary for proving our formula, although further analysis of them might help to improve its efficiency.

\begin{figure}
    \centering
    \begin{tikzpicture}
\usetikzlibrary{decorations.pathreplacing}
\draw[step=1.0,black,thin, dotted] grid (14,7);
\draw[thick, dotted] (0,7)--(1,0)--(2,6)--(3,0)--(4,1)--(5,0)--(6,2)--(7,0)--(8,3)--(9,0)--(10,5)--(11,0)--(12,4)--(13,0)--(14,7);
\fill(0,7) circle(.1);
\draw(0,7.3) node{$s_7$};
\fill(2,6) circle(.1);
\draw(2,6.3) node{$s_6$};
\fill(4,1) circle(.1);
\draw(4,1.3) node{$s_1$};
\fill(6,2) circle(.1);
\draw(6,2.3) node{$s_2$};
\fill(8,3) circle(.1);
\draw(8,3.3) node{$s_3$};
\fill(10,5) circle(.1);
\draw(10,5.3) node{$s_5$};
\fill(12,4) circle(.1);
\draw(12,4.3) node{$s_4$};
\fill(14,7) circle(.1);
\draw(14,7.3) node{$s_7$};
\draw[thick] (2,6) -- (1,0) -- (1/7,6);
\draw[thick] (1/7,6) -- (2,6);
\draw[thick] (4,1) -- (3,0) -- (2+5/6,1);
\draw[thick] (2+5/6,1) -- (4,1);
\draw[thick] (4,1) -- (5,0) -- (5+1/2,1);
\draw[thick] (4,1) -- (5+1/2,1);
\draw[thick] (6,2) -- (7,0) -- (7+2/3,2);
\draw[thick] (6,2) -- (7+2/3,2);
\draw[thick] (8,3) -- (9,0) -- (9+3/5,3);
\draw[thick] (8,3) -- (9+3/5,3);
\draw[thick] (12,4) -- (11,0) -- (10+1/5,4);
\draw[thick] (12,4) -- (10+1/5,4);
\draw[thick] (12,4) -- (13,0) -- (13+4/7,4);
\draw[thick] (12,4) -- (13+4/7,4);
\draw(1,-0.3) node{$6_l$};
\draw(3.0,-0.3) node{$1_l$};
\draw(5.0,-0.3) node{$1_r$};
\draw(7,-0.3) node{$2_r$};
\draw(9.0,-0.3) node{$3_r$};
\draw(11,-0.3) node{$4_l$};
\draw(13,-0.3) node{$4_r$};
\draw [decorate,decoration={brace,amplitude=5pt},xshift=-5pt,yshift=0pt]
(0,0) -- (0,0.9) node [black,midway,xshift=-0.9cm,yshift=0.1cm] {\footnotesize $n_0$} node [black,midway,xshift=-0.9cm,yshift=-0.1cm] {\footnotesize elements};
\draw [decorate,decoration={brace,amplitude=5pt},xshift=-5pt,yshift=0pt]
(0,1) -- (0,1.9) node [black,midway,xshift=-0.9cm,yshift=0.1cm] {\footnotesize $n_1$} node [black,midway,xshift=-0.9cm,yshift=-0.1cm] {\footnotesize elements};
\draw [decorate,decoration={brace,amplitude=5pt},xshift=-5pt,yshift=0pt]
(0,2) -- (0,2.9) node [black,midway,xshift=-0.9cm,yshift=0.1cm] {\footnotesize $n_2$} node [black,midway,xshift=-0.9cm,yshift=-0.1cm] {\footnotesize elements};
\draw [decorate,decoration={brace,amplitude=5pt},xshift=-5pt,yshift=0pt]
(0,3) -- (0,3.9) node [black,midway,xshift=-0.9cm,yshift=0.1cm] {\footnotesize $n_3$} node [black,midway,xshift=-0.9cm,yshift=-0.1cm] {\footnotesize elements};
\draw [decorate,decoration={brace,amplitude=5pt},xshift=-5pt,yshift=0pt]
(0,4) -- (0,4.9) node [black,midway,xshift=-0.9cm,yshift=0.1cm] {\footnotesize $n_4$} node [black,midway,xshift=-0.9cm,yshift=-0.1cm] {\footnotesize elements};
\draw [decorate,decoration={brace,amplitude=5pt},xshift=-5pt,yshift=0pt]
(0,5) -- (0,5.9) node [black,midway,xshift=-0.9cm,yshift=0.1cm] {\footnotesize $n_5$} node [black,midway,xshift=-0.9cm,yshift=-0.1cm] {\footnotesize elements};
\draw [decorate,decoration={brace,amplitude=5pt},xshift=-5pt,yshift=0pt]
(0,6) -- (0,6.9) node [black,midway,xshift=-0.9cm,yshift=0.1cm] {\footnotesize $n_6$} node [black,midway,xshift=-0.9cm,yshift=-0.1cm] {\footnotesize elements};
\end{tikzpicture}
\caption{Example of a pinnacle set ordering $[\tau] = [7612354]$ with corresponding dales.}
\label{DalFig}
\end{figure}
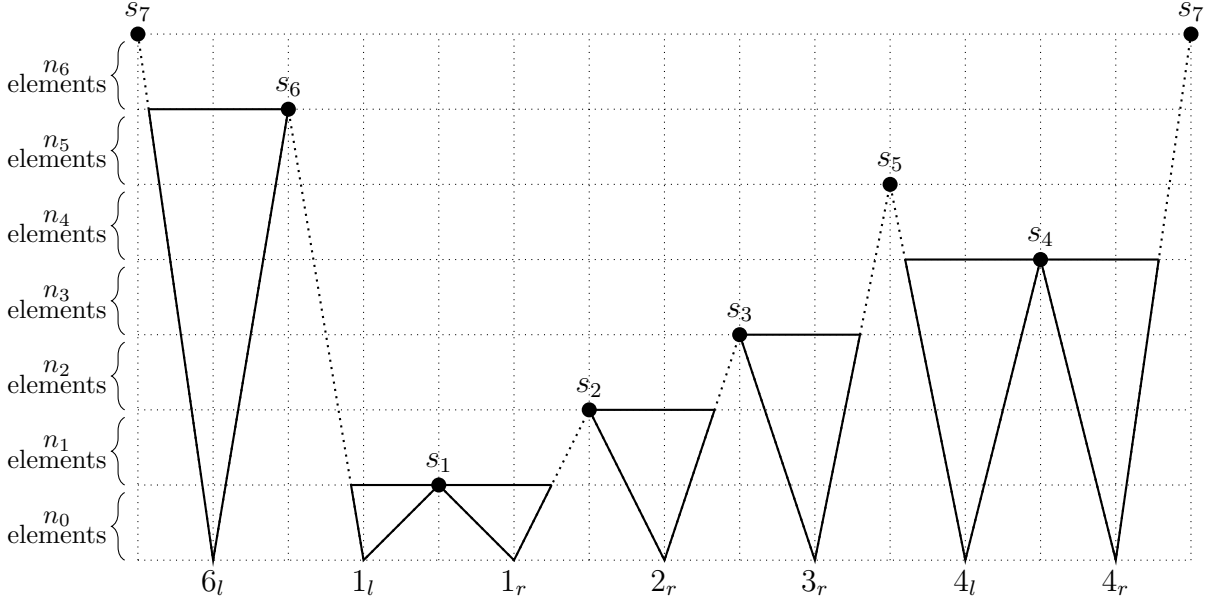

\begin{lem}
\label{LinCir}
For $n\in\bbP$, a given placement will correspond to a permutation $[\pi] \in [\fS_{n+1}]$ with pinnacle set $S'$ if and only if every dale is non-empty.
\end{lem}

\begin{proof}
First, suppose $d=0$. In this case, the theorem is trivial since there are no dales.  And every placement will automatically result in only one pinnacle, namely $n+1$, as long as $n>0$.

Now suppose $d>0$. Clearly if any dale of rank $i$ (whether left or right) is empty, then the pinnacle $s_i$ will have no smaller elements between itself and the higher pinnacle next to it, which will force $s_i$ to not be a pinnacle. On the other hand, if all dales have at least one element, then the space between any two pinnacles will always contain an element smaller than both, and all elements of $S'$ will in fact be pinnacles.
\end{proof}

We can now enumerate all placements corresponding to a given cyclic ordering of the indices of the pinnacle set $S'$.

\begin{lem}
\label{PlaFixOrd}
Given an admissible pinnacle set $S'$, fix an order $[\tau]$ of the pinnacle indices. The total number of placements with order $[\tau]$ that will result in a permutation with pinnacle set $S'$ is given by 
$$
\sum_{B \sbs D_{{[\tau]}}} (-1)^{b}\prod_{i=0}^{d} 2^{n_i}(d+1-i-b_{i+1})^{n_i}
$$ 
where $b$, $d$, the $b_i$, and the $n_i$ are defined above.
\end{lem}

\begin{proof}
We will use the Principle of Inclusion and Exclusion or PIE. We let our universal set be all possible placements with no restrictions. We then wish to exclude any placement where at least one dale is empty. Therefore, if $B$ is some subset of the dales, we must be able to count the number of placements where all dales in $B$ (and possibly others) are empty.

First consider the case when $B=\emp$.  There are $2(d+1)$ factors of which $2i$ only exist below $s_i$.  So each of the $n_i$ non-pinnacles between $s_i$ and $s_{i+1}$ may be placed in any of the $2(d+1-i)$ factors that are long enough to extend above  $s_i$. 
As an example, in Figure~\ref{LinCir} if we look between the horizontal boundary lines for the elements counted by $n_2$ we see there are $10=2(6+1-2)$ such factors represented by the diagonal lines (solid or dotted) which intersect the region.

For non-empty $B$, each dale of rank at least $i+1$ that we require to be empty will result in a loss of two additional factors, and so there are only $2(d+1-i-b_{i+1})$ choices. Therefore, for a given $B$, the total number of placements guaranteeing the dales in $B$ are empty is 
$$
\prod_{i=0}^{d} 2^{n_i}(d+1-i-b_{i+1})^{n_i}.
$$

To use the PIE, we must also attach the sign $(-1)^{|B|} = (-1)^{b}$ to this term before summing. Therefore, given a fixed order $[\tau]$ of the pinnacle indices of $S'$, we have that the total number of placements that will result in a permutation with pinnacle set $S'$ is
$$\sum_{B \subseteq D_{[\tau]}} (-1)^{b}\prod_{i=0}^{d} 2^{n_i} (d+1-i-b_{i+1})^{n_i}.$$
Finally, when $B=D_{[\tau]}$ then $b_1=\#B=d+1$.  So we can ignore this term because the product has a factor of $d+1-b_1=0$.
\end{proof}

The above formula must be summed over all possible $[\tau]$ to give a final count for the number of $[\pi]$ with $\Pin[\pi]=S'$. This  results in computationally expensive double sum.  Also,  note that in the above formula there may be multiple $B$  resulting in the same term. For example, 
$\{1_l, 2_r, 5_l\}$ is not the same as $\{1_r, 2_r, 5_l\}$ even though both produce the same $b_i$. We will take care of this redundancy when we optimize our formula below.

To fix the double sum problem, note that each $B$ in Lemma~\ref{PlaFixOrd} is a subset of the master dale rank set $D$. We will fix some subset $B \subseteq D$ and count the number of orderings $[\tau]$ that will produce a $D_{[\tau]}$ which can have $B$ as a subset. This will allow us to just sum over all subsets $B \subseteq D$ without having to keep track of $[\tau]$. Furthermore, we only have to sum over the subsets $B$ of cardinality at most $d$ since requiring more than $d$ dales to be empty is impossible for an admissible pinnacle set.

\begin{lem}
\label{[tau]forB}
Fix some $B\sbe D$ with $|B|\leq d$. The number of orderings $[\tau]$ that will produce a $D_{[\tau]}$ such that $B \sbe D_{[\tau]}$  is given by 
$$(d-b)! \prod_{i=0}^{b-1} \left(d+1-i-r_{b-i}\right)$$
where $b$, $d$, and the $r_i$ are defined as above.
\end{lem}

\begin{proof}
We will start by viewing all $d+1$ pinnacles as separate and then adjoin them in pairs in such a way so that the desired dales are formed.  Here, ``adjoining a pair of pinnacles" means requiring that they be adjacent in $[\tau]$.

We start with the dale of rank $r_b$ the largest rank in $B$.
In that case, the only way to generate such a dale is to order $s_{r_b}$ so that one of the $d+1 - r_b$ higher pinnacles is directly to its left or right depending on whether the corresponding element of $B$ is a left or right rank, respectively. So select one such pinnacle and adjoin it to the appropriate side of $s_{r_b}$. 

Next we will examine the dale in $B$ with the next highest rank, $r_{b-1}$. If $r_{b-1}$ is a smaller rank than $r_{b}$, we may once again select a taller pinnacle to place next to $s_{r_{b-1}}$, on either the left or right as necessary, in order to produce the desired dale. This time however, although there are $d+1 - r_{b-1}$ pinnacles higher than $s_{r_{b-1}}$, one of them is unavailable since we have already adjoined two of the higher-ranked pinnacles together. More specifically, because of adjoining a higher pinnacle with $s_{r_b}$, we know that one taller pinnacle cannot be joined to its left and another cannot be joined to  its right.  So no matter whether $r_{b-1}$ corresponded to a left or right dale, there is one less option. Therefore, the number of ways to append a larger pinnacle is $d+1 - r_{b-1}-1$. On the other hand, if $r_{b-1}=r_b$ then we need to adjoin a second pinnacle to $s_{r_b}$ on the side opposite the one used when considering $r_b$.  Again, the pinnacle already adjoined to $s_{r_b}$ removes one option so the number of choices is $d+1 - r_{b-1}-1$ as before. 
So in either case we have the same number of possibilities. Similar consideration show that, in general, each $r_{b-i}$ results in $d+1-i-r_{b-i}$ choices for adjoining pinnacles. Note that for this argument we are using the fact that $b\le d$ since if $b=d+1$ then the string of pinnacles would wrap into a circle before creating the final dale.

Once all dales have been created by the above process, we  only need to count the number of ways to join the resulting strings of pinnacles together. Since we have adjoined pinnacles together $b$ times, we have $d+1-b$ strings  which we then must arrange in a circle.  This can be done in $(d+1-b-1)! = (d-b)!$ ways. Therefore, 
$$(d-b)! \prod_{i=0}^{b-1} \left(d+1-i-r_{b-i}\right)$$
is the number of orderings $[\tau]$ that will allow for a given $B$ to be a subset of $D_{[\tau]}$.
\end{proof}

We are now in a position to prove Theorem~\ref{main1} which we restate here for ease of reference.

\begin{thm}
\label{p_S(n)}
Given $n\in\bbP$ and admissible  $S=\{s_1<s_2<\ldots<s_d\}$ we have
$$
p_S(n)=2^{n-2d-1}\sum_{B \subseteq D:\ |B|\leq d} (-1)^{b} (d-b)! \left(\prod_{i=0}^{b-1} \left( d+1-i-r_{b-i}\right)\right) \left(\prod_{i=0}^{d} (d+1-i-b_{i+1})^{n_i}\right).
$$
\end{thm}

\begin{proof}
It is easy to verify the formula if $d=0$, so we assume $d>0$.
From Lemma~\ref{LinCir}, the number of permutations $\pi \in \fS_{n+1}$ with pinnacle set $S$ equals the number of cyclic permutations $[\pi] \in [\fS_{n+1}]$ with pinnacle set $S' = S \cup\{n+1\}$.  So we will count the latter. From Lemma~\ref{PlaFixOrd}, the number of placements which correspond to a cyclic permutation with pinnacle set $S'$ is given by
$$\sum_{[\tau]}\ \sum_{B \subseteq D_{{[\tau]}}} (-1)^{b}\prod_{i=0}^{d}2^{n_i} (d+1-i-b_{i+1})^{n_i}$$ 
where the outer sum is over all possible cyclic orderings $[\tau]$ of the index set of $S'$.  We now wish to swap the summations so that the outer sum is over all $B \subseteq D$ with $|B|\leq d$.  We may restrict to size at most $d$ since any larger $B$ will either consist of a combination of dales that cannot exist, or will require all $d+1$ dales to be empty which is impossible because of the assumption that $d>0$.
In order to interchange the summations we must multiply the term corresponding to each $B$ by the number of distinct permutations $[\tau]$ that could have generated it. This was counted in Lemma \ref{[tau]forB}, and so we get the formula 
$$\sum_{B \sbs D:\ |B|\leq d} (-1)^{b} (d-b)! \left(\prod_{i=0}^{b-1} \left( d+1-i-r_{b-i}\right)\right) \left(\prod_{i=0}^{d} 2^{n_i} (d+1-i-b_{i+1})^{n_i}\right)$$ for the number of placements.

Now we seek to turn the placements into permutations. Since all dales are guaranteed to be non-empty, we have that every permutation corresponding to one of these placements will have $d+1$ non-pinnacle elements that are part of both a decreasing factor and an increasing factor. This means that every such corresponding $[\pi]$ has been counted by $2^{d+1}$ placements. Dividing by this, and also pulling some common factors of two out from the second product, we have
\begin{align*}
p_S(n)&= 2^{-d-1} \prod_{i=0}^{d} 2^{n_i}\sum_{B \subseteq D:\ |B|\leq d} (-1)^{b} (d-b)! \left(\prod_{i=0}^{b-1} \left(d+1-i-r_{b-i}\right)\right) \left(\prod_{i=0}^{d} (d+1-i-b_{i+1})^{n_i}\right)\\
&= 2^{n-2d-1}\sum_{B \subseteq D:\ |B|\leq d} (-1)^{b} (d-b)! \left(\prod_{i=0}^{b-1} \left( d+1-i-r_{b-i}\right)\right) \left(\prod_{i=0}^{d} (d+1-i-b_{i+1})^{n_i}\right)   
\end{align*}
where this is the formula we set out to prove.
\end{proof}

In~\cite{dnpt:psp}, explicit formulas were given for $p_S(n)$ when $|S|\le3$.
These expressions follow easily from the previous reslt.
\begin{cor}
We have the following values for $p_S(n)$.
    \begin{enumerate}
        \item[(1)] \label{PinForm1} If $S=\emp$ then
        $$
        p_S(n)=2^{n-1}.
        $$
        \item[(2)] \label{PinForm2} If  $S=\{l\}$ where $3\le l\le n$ then 
        $$
        p_S(n)=2^{n-2}(2^{l-2} - 1).
        $$
        \item[(3)] \label{PinForm3} If $S = \{l,m\}$ where $l\ge 3$, $m\ge 5$, and $l<m\le n$, then
        $$
        p_S(n)=2^{n+m-l-5}(3^{l-1} - 2^l + 1) - 2^{n-3}(2^{l-2}-1).
        $$
    \end{enumerate}

\end{cor}
\begin{proof}
In each of the results we apply Theorem $\ref{p_S(n)}$.

(1) When $d=0$, the first product in Theorem $\ref{p_S(n)}$ is always empty and the second always equals one. Therefore, everything reduces immediately to $p_S(n)=2^{n-1}$, as desired. 

(2) When $d=1$ we have $n_0 = l-1, n_1 = n-l$. Therefore, we have the following possibilities for $B$, and the corresponding terms in the summation
\begin{itemize}
    \item $B = \emptyset: 2^{l-1}$ 
    \item $B = \{1_l\}$ or $\{1_r\}$: $-1$ 
\end{itemize}
which when substituted into the formula gives
$$
p_S(n)=2^{n-3}(2^{l-1}-2) = 2^{n-2}(2^{l-2}-1).
$$

(3)  When $d=2$ we have $n_0 = l-1, n_1 = m-l-1,$ and $n_2 = n-m$. Additionally, the first inner product will always zero out if $2_r, 2_l$ are both in $B$. Therefore, we have the following possibilities for $B$, and the corresponding terms in the summation: 
\begin{itemize}
    \item $B = \emptyset: (2)3^{l-1}2^{m-l-1}$ 
    \item $B = \{1_l\}$ or $\{1_r\}$: $(-2)2^{l-1}2^{m-l-1}$ 
    \item $B = \{2_l\}$ or $\{2_r\}$: $(-1)2^{l-1}$
    \item $B = \{1_l, 1_r\}$: $2$
    \item $B = \{1_l, 2_r\}$ or $\{1_r, 2_r\}$ or $\{1_l, 2_l\}$ or $\{1_r, 2_l\}$: $1$.
\end{itemize}
When we substitute all these  into the formula, we get 
\begin{align*}
p_S(n)
&= 2^{n-5} [(2)3^{l-1}2^{m-l-1} - (4)2^{l-1}2^{m-l-1} - (2)2^{l-1} + (2)2^{m-l-1} + 4]\\
& = 2^{n-5} [(2)3^{l-1}2^{m-l-1} - (4)2^{l-1}2^{m-l-1} + (2)2^{m-l-1}] - 2^{n-5}
[(2)2^{l-1} - 4]\\
&=2^{n+m-l-5}(3^{l-1} - 2^l + 1) - 2^{n-3}(2^{l-2}-1)
\end{align*}
as desired. 
\end{proof}

We can make Theorem~\ref{p_S(n)} more efficient by summing over certain weak compositions rather than subsets.  A {\em weak composition} of $n\in\bbN$ is a sequence $\al=[\al_1,\al_2,\ldots,\al_k]$ of nonnegative integers called {\em parts} such that $\sum_i \al_i =n$.  In this case we write $\al\comp n$ or $|\al|=n$ where $|\al|=\sum_i\al_i$.  To $B \subseteq D$ we associate the composition $\alpha = [\alpha_1, \alpha_2, \ldots, \alpha_d]$ where $\alpha_i$ is the number of dales in $B$ of rank $i$.   To illustrate, for the example in Figure~\ref{DalFig} the corresponding composition is $\al=[2,1,1,2,0,1]$.  Note that all the necessary parameters for $D$ can be read off of $\al$.  In particular
$$
r_j =\min\{i \mid \al_1+\al_2+\cdots+\al_i\ge j\},
$$
and
$$
b_i = \al_i+\al_{i+1}+\cdots+\al_d.
$$
Note that
$$
b=b_1=|\al|.
$$
Thus we will be able to sum over the following set
$$
C(d) = \{\al=[\al_1,\al_2,\ldots,\al_d]\ \mid\
\text{$\al_i\in[0,2]$ for all $i$ and $|\al|\le d$}\}.
$$

We must find how many $B$ correspond to a given $\al$.  If $\al_i=0$ then $B$ contains no dales of rank $i$.  If $\al_i=2$ then $B$ contains both dales of rank $i$.  So the only choice comes if $\al_i=1$ in which case $B$ could contain either $i_l$ or $i_r$.  Letting
$$
o = \text{the number of $\al_i=1$}
$$
we see that the number of $B$ represented by $\al$ is $2^o$.  Thus we have proved the following result.
\begin{cor}
\label{p_S(n)C}
Given $n\in\bbP$ and admissible  $S=\{s_1<s_2<\ldots<s_d\}$ we have
$$
p_S(n)=2^{n-2d-1}\sum_{\alpha \in C(d)} (-1)^{b} 2^o (d-b)! \left(\prod_{i=0}^{b-1} \left( d+1-i-r_{b-i}\right)\right) \left(\prod_{i=0}^{d} (d+1-i-b_{i+1})^{n_i}\right).\qed
$$
\end{cor}

In order to compare this formula to the one in~\cite{dhhin:fep}, we need to introduce some notation.  The {\em vale set} of a permutation $\pi$ is
$$
\Val\pi=\{\pi_i \mid \pi_{i-1}>\pi_i<\pi_{i+1}\}.
$$
Call a pair $(S,T)$ {\em $n$-admissible} if there is a permutation $\pi\in\fS_n$ with $\Pin\pi=S$ and $\Val\pi=T$.  Define
$$
\cV_n(S) = \{T \mid \text{$(S,T)$ is $n$-admissible}\}.
$$
\bth[\cite{dhhin:fep}]
\label{p_S(n)V}
Given $n\in\bbP$ and admissible  $S$ with $\#S=d$ we have
$$
p_S(n)=2^{n-2d-1}\sum_{T\in\cV_n(S)}\prod_{s\in S}\binom{N_{ST}(s)}{2}
\prod_{t\in [n]-(S\uplus T)} N_{ST}(t)
$$
where $S_i = \{s\in S \mid s<i\}$,
$T_i = \{t\in T \mid t<i\}$, and
$N_{ST}(i)=\#T_i-\#S_i$.\hqed
\eth
In order to estimate the number of terms in this sum, we need a formula for $\#\cV_n(S)$.  Let
$$
K(d) =  \{\al=[\al_1,\al_2,\ldots,\al_d]\comp d\ \mid\
\text{$\al_1+\al_2+\cdots+\al_k\ge k$ for all $k\in[d]$}\}.
$$
\bth[\cite{dhhin:fep}]
\label{V_n(S)}
Given $n\in\bbP$ and admissible  $S=\{s_1<s_2<\ldots<s_d\}$ we have

\vs{5pt}

\eqed{
\#\cV_n(S) = \sum_{\al\in K(d)} \binom{n_0-1}{\al_1}\prod_{i=2}^d \binom{n_{i-1}}{\al_i}.
}
\eth

\begin{table}[]
    \centering
    $$
    \barr{c|c|c}
      S  &  \text{DLHHIN}  & \text{DLMSSS} \vsp\\
      \hline  \hline
\{3, 5, 7, 9, 11, 13, 15, 17, 19, 21\} &  9.2 \times 10^{-5} &0.72\vsp\\
\{3, 6, 9, 12, 15, 18, 21, 24, 27, 30\} & 0.11      & 0.73\vsp\\
\{3, 7, 11, 15, 19, 23, 27, 31, 35, 39\} & 9.5       & 0.73\vsp\\
\{3, 8, 13, 18, 23, 28, 33, 38, 43, 48\} & 210      & 0.78\vsp
    \earr
    $$
 
    \caption{Run times in seconds compared when most $n_i$ are equal}
    \label{RunEq}
\end{table}

We can now compare the number of terms in the sums of Corollary~\ref{p_S(n)C} and Theorem~\ref{p_S(n)V}.  In the former we have $c(d) :=\# C(d)\le 3^d$ terms, where the inequality comes from the fact that every $\al_i\in\{0,1,2\}$.  In the latter, we have $v_n(S):=\#\cV_n(S)$ terms which depends on $n$ and $S$, and not just $d$ as seen in Theorem~\ref{V_n(S)}.  If $n_1\le 4$ and $n_i\le 3$ for $i\ge 2$ then each of the binomial coefficients in the sum is a most $3$ and so $v_n(S)$ could be significantly smaller than $c(d)$.  But if even one of the $n_i$ is large, then the inequality will be reversed.  For example, suppose 
$n_1\ge 2d+1$ and take $\al=[d,0,0,\ldots,0]\in K(d)$.  Then, by Stirling's approximation,
$$
v_n(S)\ge\binom{2d}{d}\sim \frac{4^d}{\sqrt{\pi d}}
$$
which will eventually be greater than $3^d$.   So, for fixed $d$, there are only finitely many $n$ such that $v_n(S)\le c(d)$.  Thus, in most cases, Corollary~\ref{p_S(n)C} will be more efficient.  We should mention that Diaz-Lopez, Insko, and Nilsen~\cite{din:po} have come up with a refinement of the ideas in~\cite{dhhin:fep} which permits the product of binomial coefficients in Theorem~\ref{V_n(S)} to be replaced by $2^d$.

The observations of the previous paragraph are borne out by actual computer computations.
In Tables~\ref{RunEq} and~\ref{RunCo} we show the results of computing $p_S(1000)$ for various sets $S$ (first column) with constant $d$ by the algorithm in~\cite{dhhin:fep} (second column) and our algorithm (third column).  The run times are in seconds and are the average over 10 trials for each set using a 15-inch 2017 MacBook Pro with a 3.1 GHz Quad-Core Intel Core i7 processor.  In Table~\ref{RunEq} the $n_i$ for $0<i<d$ are constant in each set, but allowed to increase as one goes down the table.  As expected, the  algorithm using vales starts out orders of magnitude faster than the one using dales but quickly becomes orders of magnitude slower, with the latter's times being virtually constant.  Similar behaviour is shown in the two parts of Table~\ref{RunCo} which keep all of the $n_i$ for $0\le i < d$ constant except for one which is allowed to grow.  Note the difference in growth rate of the vale algorithm between increasing  $n_4$ (upper chart) and $n_0$ (lower chart).  

\begin{table}[]
    \centering

    $$
    \barr{c|c|c}
    \multicolumn{3}{c}{\text{Increase $n_4$ with other $n_i$ constant}}\\[5pt]
      S  &  \text{DLHHIN}  & \text{DLMSSS} \vsp\\
      \hline  \hline
\{3, 5, 7, 9, 11\} &  2.9 \times 10^{-5} &0.0014\vsp\\
\{3, 5, 7, 9, 21\} & 7.1 \times 10^{-5}  & 0.0014\vsp\\
\{3, 5, 7, 9, 31\} & 0.00012            & 0.0015\vsp\\
\{3, 5, 7, 9, 41\} & 0.00017            & 0.0015\vsp
    \earr
    $$   
    
        $$
    \barr{c|c|c}
    \multicolumn{3}{c}{\text{Increase $n_0$ with other $n_i$ constant}}\\[5pt]
      S  &  \text{DLHHIN}  & \text{DLMSSS} \vsp\\
      \hline  \hline
\{3, 5, 7, 9, 11\} &  2.9 \times 10^{-5} &0.0014\vsp\\
\{13, 15, 17, 19, 21\} & 0.012           & 0.0015\vsp\\
\{23, 25, 27, 29, 31\} & 0.26           & 0.0015\vsp\\
\{33, 35, 37, 39, 41\} & 1.8            & 0.0015\vsp
    \earr
    $$   
    
    \caption{Run times in seconds compared when most $n_i$ are  constant}
    \label{RunCo}
\end{table}

Another advantage to this approach is that it can be modified to count 
$\#\cO(S)$, the number of admissible orderings of an admissible pinnacle set $S$.
First, if we fix $n>0$ we have that Lemma \ref{circ} will again allow us to reduce to the case of cyclic orderings of the pinnacle set $S'$ for permutations in $\fS_{n+1}$. We now prove the following intermediate result. 
\begin{lem}
\label{AdmOrd}
Consider a cyclic ordering $[\tau]$ with dale set $D_{[\tau]}$ and corresponding $r_j$.  The ordering is admissible if and only if 
$$
j\leq n_0 + n_1 + \cdots + n_{r_{j}-1}
$$
for all $j\in[d+1]$.
\end{lem}
\begin{proof}
Note that,  by definition of the $n_i$ and $r_i$, the right hand side of the inequality is simply the number of non-pinnacles small enough to be placed in any of the dales having rank at least $r_j$. So if for  any $j$ we have $j > n_0 + n_1 + \cdots n_{r_{j}-1}$, then there will be at least $j+1$ dales having rank at most $r_j$.  This means there would not be enough small non-pinnacle elements to fill them all. Therefore, any such ordering is not admissible. On the other hand, if we have that $j \leq n_0 + n_1 + \cdots n_{r_{j}-1}$ for all $j$, then we may always fill all the dales by placing the smallest non-pinnacle in the lowest ranked dale, and proceeding upwards. The inequalities guarantee that we will always have enough non-pinnacles to do this at every step, and so we are done.
\end{proof}

Since the problem is trivial if $d= 0$, so we may also assume that $d>0$. We also define for the master dale rank set $D$
$$
D' = D - \{1_l, 1_r\}
$$
and for any subset $B$
$$
\delta_B = \begin{cases} 
      1 & \text{if } j \leq n_0 + n_1 + \cdots +n_{r_{j}-1} \text{ for all } j\in[b],\\
      0 & \text{otherwise.} 
   \end{cases}
$$
With this notation, we can count admissible orderings.
\begin{thm}
\label{O(S)}
If  $d\in\bbP$ and $S$ is admissible then
$$
\#\cO(S)=\sum_{B \subseteq D':\ |B|= d-1} \delta_{B \cup \{1_l, 1_r\}} \prod_{i=0}^{d-2} \left( d+1-i-r_{d-1-i}\right).
$$
\end{thm}
\begin{proof}
We first wish to sum over all possible orderings, partitioned by their dales. Since every dale set for $d>0$ is guaranteed to have $d+1$ elements and contain $\{1_l, 1_r\}$, we may index the dales by taking $B\subseteq D'$ where $|B| = d-1$. We then consider the following summation
$$\sum_{B \subseteq D':\ |B|= d-1} \,\,\, \prod_{i=0}^{d-2} \left( d+1-i-r_{d-1-i}\right).$$
Clearly this sums over every possible dale set once, and the expression inside comes from Lemma \ref{[tau]forB}, which counts the number of cyclic orderings of $S'$ which have dales containing those in $B$. However, due to the restrictions placed on the size of $B$ and the comments above,  this expression will count those cyclic orderings of $S'$ which have dales equal to 
$B \cup \{1_l, 1_r\}$ instead of just a subset. Therefore, no ordering can be counted twice by two different $B$'s and so every ordering is accounted for exactly once in the above summation, making the total $d!$.

Finally, using Lemma \ref{AdmOrd}, we may exclude from this sum precisely those orderings which are not admissible by writing it as

$$\sum_{B \subseteq D':\ |B|= d-1} \delta_{B \cup \{1_l, 1_r\}} \prod_{i=0}^{d-2} \left( d+1-i-r_{d-1-i}\right).$$

This completes the proof.
\end{proof}

We may also rewrite our result in terms of compositions for a faster summation. Lemma \ref{AdmOrd} still holds as the $r_j$ are the same whether or not the dales sets are represented as compositions, but now we will need make some new definitions.  Let
$$
C'(d) = \{\al=[\al_1,\al_2,\ldots,\al_{d-1}]\comp [d-1]\ \mid\
\text{$\al_i\in[0,2]$ for all $i$}\},
$$
and if $\al\comp b$
$$
\delta_\alpha = \begin{cases}
      1 & \text{if } j \leq n_0 + n_1 + \cdots +n_{r_{j}-1} \text{ for all } j\in[b],\\
      0 & \text{otherwise.} 
   \end{cases}
$$
Also, if $\al=[\al_1,\al_2,\ldots,\al_{d-1}]$ then define
$$
2 \oplus \alpha = [2,\al_1,\al_2,\ldots,\al_{d-1}].
$$
The following result follows from Theorem~\ref{O(S)} in much the same way that Corollary~\ref{p_S(n)C} followed from Theorem~\ref{p_S(n)}.  So the proof is omitted.
\begin{cor}
If  $d\in\bbP$ and $S$ is admissible then

\vs{5pt}

\eqed{
\#\cO(S)=\sum_{\alpha \in C'(d)} \delta_{2 \oplus \alpha} 2^o \prod_{i=0}^{d-2} ( d+1-i-r_{d-1-i}).
}
\end{cor}

\medskip

{\em Acknowledgement.}  We thank Richard Stanley who, on being shown Lemma~\ref{Delta}, pointed out that $\fp(m,d)$ is a ballot number.
We also thank Alexander Diaz-Lopez and Lars Nilsen for carefully reading the paper and useful comments.

\nocite{*}
\bibliographystyle{plain}

\newcommand{\etalchar}[1]{$^{#1}$}

\end{document}